\documentclass[a4paper,10pt]{article}

\usepackage{amsmath}
\usepackage{amssymb}
\usepackage{amsopn}
\usepackage{amsthm}
\usepackage{amsfonts}
\usepackage[latin1]{inputenc}
\usepackage{mathrsfs}
\usepackage[pdftex]{graphicx}
\usepackage{a4wide}
\usepackage[pdftex=true, colorlinks=true]{hyperref}

\newcommand{\dist}{{\ensuremath{\mathrm{dist}}}}
\newcommand{\eps}{{\varepsilon}}
\newcommand{\bR}{{\mathbb R}}
\newcommand{\bZ}{{\mathbb Z}}
\newcommand{\bN}{{\mathbb N}}
\newcommand{\bE}{{\mathbb E}}

\newcommand{\cF}{{\mathcal F}}
\newcommand{\cH}{{\mathscr H}}
\newcommand{\cL}{{\mathcal L}}

\providecommand{\distt}[1]{\mbox{\dist}_{L^2(D)}(#1,M)}

\providecommand{\abs}[1]{\left\lvert#1\right\rvert}

\providecommand{\floor}[1]{\left\lfloor#1\right\rfloor}
\providecommand{\norm}[1]{\left\Vert#1\right\Vert}

\providecommand{\normt}[1]{\left\Vert#1\right\Vert_{L^2(D)}}
\providecommand{\normh}[1]{\left\Vert#1\right\Vert_{H^1(D)}}

\newtheorem{theorem}{Theorem}[section]
\newtheorem{proposition}[theorem]{Proposition}
\newtheorem{lemma}[theorem]{Lemma}
\begin{document}

\title{Low noise limit for the invariant measure of a multi-dimensional stochastic Allen-Cahn equation}

\author{Matthias Erbar}

\maketitle

\begin{abstract}
We study the invariant measure of a discretized stochastic Allen-Cahn equation in d+1 dimensions in the low noise limit. We consider a cuboidal domain and impose the two stable phases as boundary conditions at two opposite faces. We then take a joint limit where the temperature and the mesh of the discretization go to zero while the size of the domain grows. Our main result is that the invariant measures concentrate exponentially fast around the minimizers of the free energy functional if the domain does not grow too fast.
\end{abstract}


\section{Introduction and main result}\label{intro}
Reaction-diffusion equations are used in various contexts as a phenomenological model for the separation of phases and the evolution of phase boundaries. Additional effects not captured by the simplified reaction-diffusion model are often addressed by adding an extra noise term to the equation. In this paper we consider a stochastically perturbed Allen-Cahn equation
\begin{equation}\label{eqn_AC}
 \partial_t h~=~\Delta h - F'(h) +\sqrt{2\eps}\xi\ ,
\end{equation}
where $\xi$ is a space-time white noise and $\eps$ a small parameter. Here the reaction term involves a symmetric bistable potential $F$ with two wells of equal depth at the two stable phases $\pm1$. A typical choice would be $F(u)=\frac{1}{4}(u^2-1)^2$. The Allen-Cahn equation without noise is the gradient flow in $L^ 2$ %is intimately related to%
of the so called Van-der-Waals free energy functional:
\begin{equation}\label{eqn_energy}
\cF(h)~=~\int\left[\frac{1}{2}\abs{\nabla h(x)}^2+F(h(x))\right]dx\ .
\end{equation}
We are interested in the invariant measure of (\ref{eqn_AC}) which is informally given as 
$$\mu^ \eps(dh)~=~Z^{-1}\exp\left(-\eps^{-1}\cF(h)\right)\prod\limits_{x}dh(x)$$
with a ``flat'' reference measure $\prod dh(x)$ on the space of all configurations. From a statistical mechanics viewpoint it describes the distribution of configurations at thermal equilibrium with $\eps$ corresponding to the temperature. We are interested in the low temperature limit $\eps\to 0$ where one expects the distribution to concentrate on the minimizers of the free energy $\cF$.

While (\ref{eqn_AC}) and its invariant measure are well understood in dimension one, the picture is less clear in dimensions 2 and higher. Note that (\ref{eqn_AC}) is ill-posed in the latter case since solutions to the stochastic heat equation then take values only in Sobolev spaces of negative order on which the nonlinear potential F is a priori not defined. Accordingly, also a random field with distribution $\mu^ \eps$ cannot be defined.

In this paper we focus on the $(d+1)$-dimensional case and consider (\ref{eqn_AC}) on a cuboidal domain $(-L,L)\times(0,1)^d$ which grows as $\eps$ tends to zero. We impose boundary conditions $\pm1$ on the left resp. right faces $\{\pm L\}\times(0,1)^d$ to force the appearance of an interface in the configurations. To overcome the aforementioned problem of ill-posedness we discretize the domain by a grid whose mesh size tends to zero jointly with $\eps$. We will show show that in the limit $\eps\to 0$ a discretized version of the measure $\mu^ \eps$ concentrates exponentially fast around the minimizers of the continuum free energy functional $\cF$ if the domain does not grow too fast.

Our interest in this result is twofold. Firstly, while for finite $\eps$ the continuum random field does not exist we obtain a well defined distribution on continuum configurations in the limit. This is possible since the ultra-violet divergences in the discrete model as the mesh size goes to zero are weaker than the effect of the decreasing temperature.
Secondly, we see that the effect of the decreasing temperature, favoring concentration around minimizers of the free energy, is stronger than the entropic effect originating from considering a moderately growing domain. This behavior is more interesting here than in the one-dimensional case, as the geometry of minimizing configurations is potentially richer.

We will now introduce the discretization and give a precise statement of our results.
Consider a cuboidal domain $D_L=(-L,L)\times(0,1)^d$ in $\bR^{d+1}$ and a corresponding lattice domain
$$D_{L,a}~=~\overline{D_{L-a}}\cap a\bZ^{d+1}$$
which has mesh size $a=1/n$ for some $n\in\bN$. We denote by $\cH^{L,a}$ the set of functions $h:\overline{D_L}\to \bR$ which are piecewise linear w.r.t. the grid $D_{L,a}$ and satisfy the boundary condition $h(\pm L,\cdot)=\pm1$. More precisely, for $\tilde{h}:D_{L,a}\to\bR$ and $(x,y)\in \overline{D_L}$ we set
$$P\tilde{h}(x,y)~=~\begin{cases}
                     \tilde{h}(x,y)\ , &\ (x,y)\in D_{L,a}\ ,\\
		     \pm1\ , &\ x\geq(\floor{\frac{L}{a}}+1)a\ \mbox{ resp. }\ x\leq-(\floor{\frac{L}{a}}+1)a\ ,\\
		     \mbox{linear interpolation}\ ,&\ \mbox{else.}
                    \end{cases}
$$

For the linear interpolation we fix a way to subdivide a cube of side length $a$ in the grid into simplices. Then the set of piecewise linear functions w.r.t. $D_{L,a}$ and $\pm1$ boundary conditions is defined as
$$\cH^{L,a}~=~\left\{h:\overline{D_L}\to \bR~:~h=P\tilde{h}\ \mbox{ for some }\ \tilde{h}:D_{L,a}\to\bR\right\}\ .$$
%Here and in the following we will denote points in $\bR\times(0,1)$ by $(x,y)$ with $x\in\bR$ and $y\in (0,1)$. 
Note that $\cH^{L,a}$ can be identified with $\bR^N$ where $N(L,a)~=~(2\floor{L/a}-1)(1/a+1)^d\ .$
On $\cH^{L,a}$ we introduce the measure
$$\mu^{L,a,\eps}(dh)~=~\frac{1}{Z^{L,a,\eps}}\exp\left(-\frac{1}{\eps}\int\limits_{D_L}\frac{1}{2}\abs{\nabla h(z)}^2+F(h(z))dz \right)\cL^N(dh)\ ,$$
where $\cL^N$ denotes the Lebesgue measure on $\cH^{L,a}$ and $Z^{L,a,\eps}$ is a normalizing constant. Note that $\mu^{L,a,\eps}$ can be identified as the invariant measure of a suitable discretization of the Allen-Cahn equation (\ref{eqn_AC}) on the grid.

We assume that the potential satisfies the following properties. $F$ is $C^3$, symmetric and nonnegative such that $F(u)=0$ iff $u=\pm1$, $F'(u)=0$ iff $u=0,\pm1$ and $F''(0)<0,F''(\pm1)>0$.

We consider a joint limit where $\eps\to 0$ while the mesh size $a(\eps)$ tends to $0$ and the size $L(\eps)$ of the domain grows to $\infty$ as functions of $\eps$. We will show that the measures $\mu^{L,a,\eps}$ concentrate around minimizers of the free energy functional $\cF$ defined on functions on the domain $D=\bR\times(0,1)^d$ under the given boundary conditions $h(\pm\infty,y)=\pm1$.

Let us denote by $M$ the set of minimizers of $\cF$. We show that $M=\{m_\xi(x,y)=m(x-\xi)\ :\ \xi\in\bR\}$ for a characteristic transition profile $m:\bR\to[-1,1]$. Set $m_0+L^2(D)=\{h:D\to\bR\ \vert\ h-m_0\in L^2(D)\}$. By extension with $\pm1$ of functions in $\cH^{L,a}$ the measures $\mu^{L,a,\eps}$ can be regarded as measures on $m_0+L^2(D)$. For a function $h\in L^2(D)$ we will write $\distt{h}:=\inf\limits_{\xi\in\bR}\normt{h-m_\xi}$. Now our main result is the following:
\begin{theorem}\label{thm_mainthm}
 Let $L(\eps)\sim\eps^{-\lambda}$ and $a(\eps)\sim\eps^\alpha$ as $\eps\to0$ with $\lambda,\alpha>0$ such that $\lambda+(d+1)\alpha<1$. Then $\mu^\eps:=\mu^{L,a,\eps}$ concentrates around $M$ exponentially fast as $\eps\to0$. More precisely, there are constants $c_0,\delta_0>0$ such that for all $\delta<\delta_0$
$$\limsup\limits_{\eps\downarrow0}\eps\log\mu^\eps\{dist_{L^2(D)}(h,M)>\delta\}~\leq~-c_0\delta^2\ .$$
\end{theorem}

Originally the Allen-Cahn equation without noise was introduced in \cite{AC79} to model the separation of domains of different lattice structure in crystals and the dynamics of interface between these domains. It has been used as a phenomenological model in various contexts since and has also been derived from the microscopic dynamic of an underlying stochastic particle system in the hydrodynamic limit \cite{Bo87}.

In two and more dimensions solutions to the deterministic equation tend quickly towards configurations which are locally equal to one of the two stable phases with diffuse interfaces in between. On a slower scale these interfaces then undergo motion by mean curvature, see \cite{Il93}.

The stochastic dynamics have been studied in one dimension e.g in \cite{Fu95}. In this case the invariant measure is absolutely continuous with respect to a Brownian bridge connecting the stable phases, see \cite{RE05}. In higher dimensions the SPDE (\ref{eqn_AC}) can be given rigorous meaning in various ways replacing for example the white noise by a smoothened noise with spatial correlations (see e.g. \cite{KOR07}). For a different approach using a renormalization procedure for the potential $F$ see \cite{PT07}. In dimension 2 and the typical case where $F$ is given as a quartic polynomial the invariant measure of (\ref{eqn_AC}) is also referred to as the Euclidean quantum anharmonic oscillator or the $\Phi^ 4_2$-measure (see also \cite{Si74}).

The analogous result to Theorem \ref{thm_mainthm} in one dimension was proven in \cite{We10} in a continuous setting but with a similar discretization argument. A similar result was obtained in \cite{Be08} on an interval growing much more slowly where also convergence to a limiting measure on the minimizers could be established.

This work is based on \cite{We10} and largely employs the same techniques. While the choice of a growing prismlike domain and the particular boundary conditions still bears some one-dimensional traits Theorem \ref{thm_mainthm} can be seen as a first step in the investigation of the invariant measure in higher dimensions dealing with the problem of ill-posedness of the equation.

The rest of this paper is organized as follows. In the next section we analyze the free energy functional $F$ and derive estimates on the energy landscape in terms of tubular coordinates around the set of minimizers. In Section 3 we introduce several Gaussian measures which serve as reference measures for $\mu^{L,a,\eps}$. We give bounds and concentration properties to be used later in the proof. Finally in Section 4 we carry out the proof of Theorem \ref{thm_mainthm}.

We adopt the convention that $C$ denotes a generic constant whose value may change from line to line. Constants that appear several times will be numbered $c_1,c_2,$ etc.

\section{Analysis of the free energy functional}\label{sec:2}

In this section we will determine the  minimizers of the Allen-Cahn free energy under the prescribed boundary conditions and introduce tubular coordinates in a neighborhood of this set. Then we will derive estimates on the free energy landscape crucial for the rest of the paper in terms of these coordinates.\\
For a function $h$ on $D=\bR\times(0,1)^d$ satisfying the boundary condition $h(\pm\infty,y)=\pm1$ we define the free energy
\begin{equation}\label{eqn_freeenergy}
\cF(h)~=~\int\limits_D\frac{1}{2}\abs{\nabla h}^2+F(h)dz~-~C_*\ ,
\end{equation}
where $C_*$ is chosen such that the minimum of $\cF$ over functions with the right boundary values is zero.
Let us first determine the set of minimizers of this functional. By completing the square we obtain
\begin{align*}\label{eqn_completesquare}
 \int\limits_D\frac{1}{2}\abs{\nabla h}^2+F(h)dz~&=~\int\limits_\bR\int\limits_{(0,1)^d}\frac{1}{2}\abs{\nabla_y h}^2+\frac{1}{2}\left(\partial_x h-\sqrt{2F(h)}\right)^2 + \partial_x h \sqrt{2F(h)}dydx\\
 &\geq~\int\limits_{(0,1)^d}\int\limits_{\bR}\partial_x h \sqrt{2F(h)}dxdy~=~%\int\limits_{(0,1)^{d}}\int\limits_{h(-\infty,y)}^{h(\infty,y)}\sqrt{2F(u)}dudy~=~
\int\limits_{-1}^1\sqrt{2F(u)}du~:=~C_*\ .
\end{align*}
We have equality if and only if $\nabla_yh\equiv0$ and $\partial_xh=\sqrt{2F(h)}$. By our assumptions on $F$ the one-dimensional equation 
$$m'(x)-\sqrt{2F(m(x))}~=~0, \quad m(\pm\infty)~=~\pm1$$
has a unique solution $m$ with $m(0)=0$ and all other solutions are obtained via translation. We conclude that the set of minimizers of $\cF$ is given as
$$M~=~\left\{m_\xi(x,y)=m(x-\xi)\ \vert\ \xi\in \bR\right\}$$
and thus consists of the minimizers of the one-dimensional Ginsburg-Landau energy functional trivially extended to the $(d+1)$-dimensional domain. In the case of our example $F(u)=\frac{1}{4}(u^2-1)^2$ one finds $m=\tanh(\cdot/\sqrt{2})$.

From the assumptions on $F$ it follows that $m$ converges exponentially fast to $\pm1$, i.e. there exist positive constants $c_1$ and $c_2$ such that
\begin{equation}\label{eqn_tail}
 \begin{cases}
| 1 \mp m( \pm s)| \leq c_1  \exp(-c_2  s) \qquad & s \geq 0\\
|m'(\pm s)| \leq c_1 c_2  \exp(-c_2  s) \qquad & s \geq 0\\
|m''(\pm s)| \leq c_1 c_2^2  \exp(-c_2  s) \qquad & s \geq 0\ .
\end{cases}
\end{equation}
This shows that $m_\xi-m_{\xi'}\in L^2(D)$ for all $\xi,\xi'$ and so $m+L^2(D)$ is independent of the choice of a particular minimizer. Note that if $\distt{h}$ is small enough there exist a unique $\xi\in\bR$ such that $\distt{h}=\normt{h-m_\xi}$ and one has
\begin{equation}\label{eqn_normal}
\langle h-m_\xi, \partial_xm_\xi \rangle_{L^2(D)}~=~0\ .
\end{equation} 
For $h=m_\xi+v$ with $\langle v, \partial_xm_\xi \rangle_{L^2(D)}=0$ we will call the pair $(\xi,v)$ the tubular coordinates of $h$. We now state a proposition that characterizes the behavior of the energy functional close to the set of minimizers.
\begin{proposition}\label{prop_energylandscape}
 \begin{enumerate}
  \item[i)] There are constants $c_3,\delta_3>0$ such that for all $h\in m+H^1(D)$ with coordinates $h=m_\xi+v$ and $\normt{v}\leq\delta_3,\norm{v}_{L^\infty(D)}\leq 1$ we have that
$$\cF(h)~\leq~c_3\norm{v}^2_{H^1(D)}\ .$$
  \item[ii)] There are constants $c_0, \delta_0>0$ such that for all $\delta\leq\delta_0$ and $h:D\to\bR$ piecewise differentiable with $\distt{h}>\delta$ we have that
$$\cF(h)~>~c_0\delta^2\ .$$ 
 \end{enumerate}
\end{proposition}
The differentiability condition on $h$ is no severe restriction as the proposition will later only be applied to piecewise linear functions.
\begin{proof}[Proof of i] Let $\delta_3$ be such that tubular coordinates are defined for all $h$ satisfying \\$\distt{h}<\delta_3$ and let $h=m_\xi+v$ with $\normt{v}\leq\delta_3,\norm{v}_{L^\infty(D)}\leq 1$.
Using $\Delta m_\xi=F'(m_\xi)$ we write:
\begin{equation*}
 \begin{split}
  \cF(h)~&=~\int\frac{1}{2}\abs{\nabla m_\xi+\nabla v}^2 + F(m_\xi +v)dz -\int\frac{1}{2}\abs{\nabla m_\xi}^2 + F(m_\xi)dz\\
         &=~\int\frac{1}{2}\abs{\nabla v}^2 + \frac{1}{2}F''(m_\xi)v^2dz\\
         &~~+ \int F(m_\xi+v)-F(m_\xi)-F'(m_\xi)v-\frac{1}{2}F''(m_\xi)v^2dz~.
 \end{split}
\end{equation*}
Using Taylor expansion of $F$ around $m_\xi$ in the second integral we can estimate:
\begin{equation*}
 \begin{split}
  \cF(h)~&\leq~\frac{1}{2}\sup\limits_{u\in[-1,1]}\abs{F''(u)}\cdot\normh{v}^2 + \sup\limits_{u\in[-2,2]}\abs{F'''(u)}\cdot\norm{v}_{L^\infty(D)}\cdot\normt{v}^2\\
         &\leq~c_3\normh{v}^2\ .
 \end{split}
\end{equation*}
\end{proof}

 The proof of ii) is based on a corresponding one-dimensional result from \cite{We10} (similar estimates were already obtained in \cite{Fu95}, \cite{OR07}) and an induction argument on the dimension $d$. As the dimension $d$ is implicit when writing $\cF$ or $D$. So let us introduce the following notation for the moment to explicitly keep track of the dimension. Let $D^n=\bR\times(0,1)^n$ and define for $u:D^n\to\bR$ the $(n+1)$-dimensional energy functional
$\cF_n(u)=\int_{D^n}\frac{1}{2}\abs{\nabla u}^2+F(u)dz-C_*$. Under the boundary condition $u(\pm\infty,\cdot)=\pm1$ the set of minimizers of $\cF_n$ is given by $M_n=\{m_\xi(x,y)=m(x-\xi)\ \vert \ \xi\in\bR\}$. For $h:D^{n+1}\to\bR$ we write $h_t$ for the function $w\mapsto h(w,t)$ with $w\in D^n,\ t\in(0,1)$. We have the following
\begin{proposition}{\cite[Prop. 2.2]{We10}}\label{prop_energylandscape1D}
There are constants $c_4, \delta_4>0$ such that for all $\delta\leq\delta_4$ and $u:\bR\to\bR$ with $\dist_{H^1(\bR)}(u,M_0)>\delta$ we have
$$\cF_0(h)~>~c_4\delta^2\ .$$ 
\end{proposition}
\begin{proof}[Proof of Proposition \ref{prop_energylandscape} ii]
We argue by induction on $d$. For $d=0$ the assertion is true by Proposition \ref{prop_energylandscape1D}. So assume that the assertion is true for $d=n$ with a constant $c$ and let us proof it for $d=n+1$. The idea is as follows:
We want to show that a function $h:D^n\times(0,1)\to\bR$ cannot have arbitrarily small energy while keeping a distance $\delta$ to the set $M_{n+1}$. Suppose there was such a function. As changes of $h$ in direction of the last coordinate cost energy, Lemma \ref{lem_distL2} below shows that the bound on the energy implies that the functions $h_t=h(\cdot,t)$ have to keep a certain distance uniform in $t$ from the minimizers $M_{n}$. Then we can invoke the $n$-dimensional energy estimate from the induction hypothesis to conclude the proof.

So let $\delta_4$ be the constant from Proposition \ref{prop_energylandscape1D} and set $\delta_0=\delta_4$. Let $\delta\leq\delta_0$ and \\$\dist_{L^2(D^{n+1})}(h,M_{n+1})>\delta$. We further assume that $\cF_{n+1}(h)\leq\frac{1}{8}\delta^2$. From Lemma \ref{lem_distL2} we deduce that for all $t\in(0,1)$:
$$\dist_{L^2(D^{n})}(h_t,M_n)~\geq~\delta -\norm{\partial_t h}_{L^2(D^{n+1})}~\geq~\delta-\sqrt{2\cF_{n+1}(h)}~\geq~\frac{1}{2}\delta\ .$$
By the induction hypothesis we have $\cF_n(h_t)> c\frac{1}{4}\delta^2$ for all $t\in(0,1)$ which immediately implies that
$$\cF_{n+1}(h)~\geq~\int\limits_0^1\cF_n(h_t)dt~>~c\frac{1}{4}\delta^2\ .$$
Choosing $c_0=\min\{c\frac{1}{4},\frac{1}{8}\}$ we finish the proof.
\end{proof}

\begin{lemma}\label{lem_distL2}
Let $h:D^{n+1}\to\bR$ be piecewise differentiable. %For $t\in(0,1)$ write $h_t=h(\cdot,t)$.
Then we have~:
$$\dist_{L^2(D^{n+1})}(h,M_{n+1})~\leq~\min\limits_{t\in(0,1)}\dist_{L^2(D^{n})}(h_t,M_n)+\norm{\partial_th}_{L^2(D^{n+1})}\ .$$
\end{lemma}

\begin{proof}
Let $t_0\in(0,1)$ and let $\xi$ be such that $$\norm{h_{t_0}-m_\xi}_{L^2(D^{n})}~=~\dist_{L^2(D^{n})}(h_{t_0},M_n)\ .$$ 
Then we can estimate:
\begin{align*}
\dist_{L^2(D^{n+1})}(h,M_{n+1})~\leq&~\norm{h-m_\xi}_{L^2(D^{n+1})}~\\
\leq&~\left(\int\limits_0^1\left(\norm{h_t-h_{t_0}}_{L^2(D^{n})}+\norm{h_{t_0}-m_\xi}_{L^2(D^{n})}\right)^2dt\right)^{\frac{1}{2}}\\
\leq&~\left(\int\limits_0^1\norm{h_t-h_{t_0}}_{L^2(D^{n})}^2dt\right)^{\frac{1}{2}}+\dist_{L^2(D^{n})}(h_{t_0},M_n)\ .
\end{align*}
To finish the proof we note that by Jensens's inequality
\begin{align*}
 \norm{h_t-h_{t_0}}^2_{L^2(D^{n})}~=~\int\limits_{D^n}\left(\int\limits_{t_0}^t \partial_sh(w,s)ds\right)^2dw%~\leq~\int\limits_{D^n}\left(\int\limits_{0}^1 \abs{\partial_zh(x,t)}dt\right)^2dx
~\leq~\norm{\partial_th}^2_{L^2(D^{n+1})}\ .
\end{align*}
\end{proof}

We finish this section with an approximation of the minimizers $m_\xi$ by functions in $\cH^{L,a}$, where we assume that $L(\eps)\sim\eps^{-\lambda}$ and $a(\eps)\sim\eps^{\alpha}$  as $\eps\to0$ with $\lambda,\alpha>0$. Fix a $\lambda_1$ satisfying
\begin{equation}\label{eqn_gamma1}
 0~<\lambda_1~<~\min(2\alpha,\lambda)
\end{equation}
and define the function $m_\xi^\eps$ as follows.
First consider a smooth function $\tilde{m}^\eps:\bR\to\bR$ which coincides with $m$ on $[-\eps^{-\lambda_1},\eps^{-\lambda_1}]$ and which satisfies $\tilde{m}^\eps=\pm1$ on $[\eps^{-\lambda_1}+1,\infty)$ respectively $(-\infty,-\eps^{-\lambda_1}-1]$. We further assume that $m\leq\tilde{m}^\eps$ respectively $m\geq\tilde{m}^\eps$ on the sets $[\eps^{-\lambda_1},\eps^{-\lambda_1}+1)$ and $(-\eps^{-\lambda_1}-1,-\eps^{-\lambda_1}]$. We also assume that $|\frac{d}{dx}\tilde{m}^\varepsilon| \leq 2 c_1 c_2 e^{-c_2 \varepsilon^{-\lambda_1}}$ on both of these intermediate intervals. Then we set
\begin{equation*}
 m_\xi^\eps(x,y)~=~\begin{cases} \tilde{m}^\eps(x-\xi) & \mbox{if } (x,y)\in D_{L,a}\ ,\\
                                 \pm1 & \mbox{if } x\geq\eps^{-\lambda}\ (\mbox{resp. }x\leq\-\eps^{-\lambda})\ ,\\
                                 \mbox{linear interpolation} & \mbox{else~.}
\end{cases}
\end{equation*}
Then we have the following error bounds.

\begin{lemma}
 \label{lem_discbo1}
For $\varepsilon$ small enough and $\xi \in [-\varepsilon^{-\lambda}+\varepsilon^{-\lambda_1}+1,\varepsilon^{-\lambda}-\varepsilon^{-\lambda_1}-1]$ there is a constant $C>0$ such that
\begin{equation*}
 \begin{split}
  \|m_\xi-m_\xi^{\varepsilon} \|_{L^2(D)}~&\leq~ C \varepsilon^{-\lambda_1/2} \varepsilon^{2\alpha}\\
  \|\partial_x m_\xi- \partial_x m_\xi^{\varepsilon} \|_{L^2(D)}~&\leq~ C \varepsilon^{-\lambda_1/2} \varepsilon^{\alpha}.
 \end{split}
\end{equation*}
\end{lemma}

\begin{proof}
See \cite[Lemma 2.4]{We10} which immediately adapts to our higher dimensional situation. Note that $N$ from this reference corresponds to $\floor{L/a}\approx\eps^{-\lambda-\alpha}$.
\end{proof}

\section{Gaussian estimates}\label{sec:3}

In this section we introduce several finite-dimensional Gaussian measures that will serve as reference measures for the measure $\mu^{L,a,\eps}$ we want to study. We provide estimates for their normalization constants that will be needed in section \ref{sec_normalization}.

Let us first introduce a measure on $\cH^{L,a}$ which is an analogue of the discrete Gaussian free field or harmonic crystal. We set
\begin{equation*}\label{eqn_defDGFF}
\nu_1^\eps(dh)~=~\frac{1}{Z_1^\eps}\exp\left(-\frac{1}{2\eps}\int\limits_{D_L}\abs{\nabla h(z)}^2dz\right)\cL^N(dh)\ ,
\end{equation*}
where $Z_1^\eps$ is the appropriate normalization constant. Recall that $\cL^N$ denotes Lebesgue measure on the $N$-dimensional space $\cH^{L,a}$. In a similar way we define an unscaled version of the above measure:
\begin{equation*}\label{eqn_defDGFFunscaled}
\nu_2^\eps(dh)~=~\frac{1}{Z_2^\eps}\exp\left(-\frac{1}{2}\int\limits_{D_L}\abs{\nabla h(z)}^2dz\right)\cL^N(dh)\ .
\end{equation*}
Further we introduce an analogue of the discrete massive Gaussian free field. Similar to the definition of $\cH^{L,a}$ in Section 1 let us denote by $\cH^{L,a}_0$ the set of piecewise linear functions $h:D_{L}\to\bR$ having zero boundary conditions at $x=\pm L$ instead of the boundary condition $\pm1$. On $\cH^{L,a}_0$ we define the measure
\begin{equation*}\label{eqn_defMDGFF}
\rho^\eps_\kappa(dh)~=~\frac{1}{Z_3^{\eps,\kappa}}\exp\left(-\frac{\kappa}{2\eps}\int\limits_{D_L}\abs{\nabla h(z)}^2+\abs{h(z)}^2dz\right)\cL^N(dh)\ .
\end{equation*}
The following estimates are needed in the calculations in Section \ref{sec_normalization}.

\begin{lemma}\label{lem_normalizationbounds}
 We have the following bounds on the normalization constants:
 \begin{align*}
  i) & \quad \frac{Z_2^\eps}{Z_1^\eps}~=~\eps^{-\frac{N}{2}}\exp\left(\frac{1}{L}\left(\frac{1}{\eps}-1\right)\right)\ ,\\
  ii) & \quad \exp\left(\frac{1}{\eps L}\right)\kappa^{-\frac{N}{2}}(1+CL^2)^{-\frac{N}{2}}~\leq~\frac{Z_3^{\eps,\kappa}}{Z_1^\eps}~\leq~\exp\left(\frac{1}{\eps L}\right)\kappa^{-\frac{N}{2}}\ .
 \end{align*}
\end{lemma}

\begin{proof}
 Note that the integrals appearing in the density of $\nu_2^\eps$ for example can be written in the coordinates $(h_z)_{z\in D_{L,a}}$ of $\cH^{L,a}$ where $h_z=h(z)$ for $z\in D_{L,a}$. Indeed, let $l\in\cH^{L,a}$ be the linear function $l(x,y)=x/L$. Then we can write
\begin{align*}
 \int_{D_L}\abs{\nabla h}^2dz~&=~\int_{D_L}\abs{\nabla h-\nabla l}^2dz + \int_{D_L}\abs{\nabla l}^2dz \\
                               &=~\sum\limits_{z,z'\in D_{L,a}} (h_z-l_z)\Lambda_{zz'}(h_{z'}-l_{z'})+\frac{2}{L} \ ,
\end{align*}
for a suitable positive definite matrix $\Lambda\in\bR^{N\times N}$ since the first term in the first line is bilinear in $h-l$. Having represented the exponent in the density of $\nu_2^\eps$ in this form we easily calculate the normalization constant
\begin{equation*}\label{eqn_Z3}
 Z_2^\eps~=~\exp\left(-\frac{1}{L}\right)(2\pi)^{\frac{N}{2}}det(\Lambda)^{-\frac{1}{2}}\ .
\end{equation*}
Similarly we obtain
\begin{equation}\label{eqn_Z1}
Z_1^\eps~=~\exp\left(-\frac{1}{\eps L}\right)(2\pi\eps)^{\frac{N}{2}}det(\Lambda)^{-\frac{1}{2}}\ ,
\end{equation}
which proves i). Now write for $h\in\cH^{L,a}_0$ in a similar manner as above
\begin{equation}\label{eqn_H1coordinates}
\int_{D_L}\abs{\nabla h}^2+h^2dz~=~\sum\limits_{z,z'\in D_{L,a}} h_z\Lambda_{zz'}h_{z'} + \sum\limits_{z,z'\in D_{L,a}} h_zI_{zz'}h_{z'}\ ,
\end{equation}
where $I\in\bR^{N\times N}$ is another suitable positive definite matrix. Thus we calculate
\begin{equation}\label{eqn_Z2}
Z_3^{\eps,\kappa}~=~\left(\frac{2\pi\eps}{\kappa}\right)^{\frac{N}{2}}det(\Lambda+I)^{-\frac{1}{2}}\ .
\end{equation}
The Poincar\'{e} inequality shows that for all $h$:
$$\langle h,Ih\rangle~=~\int_{D_L} h^2dz~\leq~C\cdot L^2\int_{D_L}\abs{\nabla h}^2dz~=~C\cdot L^2\langle h,\Lambda h\rangle\ .$$
Hence we conclude that
$$(1+CL^2)^{-\frac{N}{2}}det(\Lambda)^{-\frac{1}{2}}~\leq~det(\Lambda + I)^{-\frac{1}{2}}~\leq~det(\Lambda)^{-\frac{1}{2}}\ .$$
Combining this with formulas (\ref{eqn_Z1}) and (\ref{eqn_Z2}) we finish the proof of ii).
\end{proof}

We finish this section by proving a concentration property of the measure $\rho^\eps_\kappa$ as $\eps$ goes to $0$.

\begin{lemma}\label{lem_concentrationMDGFF}
Let $L(\eps)=\eps^{-\lambda}$ and $a(\eps)=\eps^{\alpha}$ with $\lambda,\alpha>0$. Then there exists a constant $C>0$ such that for all $\delta, r>0$~:
 \begin{align*}
  i) & \quad \rho^\eps_\kappa\left\{h\in\cH^{L,a}_0\ : \ \norm{h}_\infty\geq\delta\right\}~\leq~N\exp\left(-\frac{\delta^2}{2C\eps^{1-2\alpha}}\right)\ ,\\
  ii) & \quad \rho^\eps_\kappa\left\{h\in\cH^{L,a}_0\ : \ \norm{h}_{H^1}\geq\sqrt{\frac{\eps N}{\kappa}}+r\right\}~\leq~\exp\left(-\frac{\kappa r^2}{2\eps}\right)\ .
 \end{align*}
\end{lemma}

\begin{proof}i) Note that if the random function $h\in\cH^{L,a}_0$ is distributed according to $\rho^\eps_\kappa$ then each coordinate $h_z=h(z)$ with $z\in D_{L,a}$ is a centered Gaussian random variable. Hence we can estimate~:
\begin{align*}
 \rho^\eps_\kappa\left\{\norm{h}_\infty\geq\delta\right\}~&\leq~\sum\limits_{z\in D_{L,a}}\rho^\eps_\kappa\left\{\abs{h_z}\geq\delta\right\}\\
&\leq~\sum\limits_{z\in D_{L,a}}\exp\left(-\frac{\delta^2}{2\bE[h_z^2]}\right)\ .
\end{align*}
Since $\rho^\eps_\kappa$ is a centered Gaussian measure with covariance operator $\frac{\kappa}{\eps}(\Lambda+I)$ we easily calculate:
\begin{equation*}\label{eqn_concentration1}
 \bE[h_z^2]~=~\frac{\eps}{\kappa}\big((\Lambda+I)^{-1}\big)_{z,z}~\leq~\frac{\eps}{\kappa}\sup\limits_{u\neq0}\frac{\langle u,(\Lambda+I)^{-1}u\rangle}{\langle u,u\rangle}
~=~\frac{\eps}{\kappa}\left(\inf\limits_{u\neq0}\frac{\langle u,(\Lambda+I)u\rangle}{\langle u,u\rangle}\right)^{-1}\ .
\end{equation*}
To finish the proof we will show that there exists a constant $C>0$ such that:
\begin{equation}\label{eqn_concentration2}
 \langle u,(\Lambda+I)u\rangle~\geq~\langle u,I u\rangle~\geq~C a^{d+1}\langle u, u\rangle\quad\mbox{for all}\ u\in\bR^N\ .
\end{equation}
Indeed, identifying $\bR^N$ again with $\cH_0^{L,a}$ we can write :
\begin{equation*}\label{eqn_concentration3}
\langle u,I u\rangle~=~\norm{u}^2_{L^2(D_L)}~=~\sum\limits_{i=1}^K\norm{u}^2_{L^2(C_{i})}\ ,
\end{equation*}
where $C_{i}$ are the $K=2\floor{\frac{L}{a}}\floor{\frac{1}{a}}^{d}$ small cubes of sidelength $a(\eps)$ of the grid $D_{L,a}$. Let $z_j$ for $j=1,\dots,2^{d+1}$ be the vertices of such a cube and $u_j=u_{z_j}$. The $L^2$-norm of $u$ on $C_i$ is bilinear in the coordinates of $u$. Hence by symmetry there are constants $b_1>0$ and $b_{j,k}$ depending only on $d$ such that
\begin{equation*}\label{eqn_concentration4}
\begin{split}
\norm{u}^2_{L^2(C_{i})}~&=~a^{d+1}\left(b_1\sum\limits_j u_j^2 + \sum\limits_{j\neq k}b_{j,k}u_ju_k\right)\\
~&=~\frac{a^{d+1}b_1}{2^{d+2}-2}\sum\limits_{j\neq k}\left(u_j^2+u_k^2+ \frac{(2^{d+2}-2)b_{j,k}}{b_1}u_ju_k\right)\ .
\end{split}
\end{equation*}
 Since the left hand side vanishes if and only if $u_j=0$ for all $j$ we must have that\\$b_2:=\max_{j\neq k}\abs{(2^{d+2}-2)b_{j,k}/b_1}<2$ and hence we obtain
\begin{equation*}\label{eqn_concentration5}
\norm{u}^2_{L^2(C_{i})}~\geq~\frac{a^{d+1}b_1}{2^{d+2}-2} \sum\limits_{j\neq k}\left(1-\frac{b_2}{2}\right)\left(u_j^2+u_k^2\right)~=~Ca^{d+1}\sum\limits_ju_j^2\ .
\end{equation*}
ii) Recall equation (\ref{eqn_H1coordinates}) which expresses the $H^1$-norm in coordinates of $\cH^{L,a}_0$. %from the proof of Lemma \ref{lem_normalizationbounds}%.
Now use the linear transformation $u=(I+\Lambda)^{1/2}h$ to write
\begin{equation*}
 \rho^\eps_\kappa\left\{h\in\cH^{L,a}_0\ \vert \ \norm{h}_{H^1}\geq R\right\}~=~\left(\frac{\kappa}{2\pi\eps}\right)^{\frac{N}{2}}\int\limits_{\{\sum\limits_z u_z^2\geq R\}}
\exp\left(-\frac{\kappa}{2\eps}\sum\limits_{z}u_z^2\right)\prod\limits_zdu_z\ .
\end{equation*}
Thus the problem is reduced to considering a centered Gaussian measure on $\bR^N$ with covariance matrix $\frac{\eps}{\kappa}Id$. Using Lemma \ref{lem_concentrationGaussian} then finishes the proof.
\end{proof}
Recall the following well known result on concentration of Gaussian measures.
\begin{lemma}\label{lem_concentrationGaussian}
 Let $\mu$ be a centered Gaussian measure on a Hilbert space $E$ with covariance operator $\Sigma$, whose spectral radius is denoted by $\sigma$. Then one has
\begin{equation*}\label{eqn_concentrationGaussian}
\mu \bigl(x: \| x \| \geq \big( \text{\emph{Tr}}\, \Sigma \big)^{1/2} +r  \bigr)~\leq~ e^{-r^2/2 \sigma^2}\ . 
\end{equation*}
\end{lemma}

\section{Proof of the main theorem}\label{sec_normalization}

To prepare the proof of Theorem \ref{thm_mainthm} we express the measures $\mu^\eps:=\mu^{L,a,\eps}$ in terms of the Gaussian measures $\nu_1^\eps$:
\begin{equation*}\label{eqn_munu}
 \mu^\eps(dh)~=~\frac{1}{Z^\eps}\exp\left(-\frac{1}{2\eps}\int\limits_{D_L}F(h(z))dz\right)\nu_1^\eps(dh)\ .
\end{equation*}
Note that the normalization constant $Z^\eps$ takes the following form:
\begin{equation}\label{eqn_zeps}
 \begin{split}
Z^{\varepsilon}~&=~\int_{\cH^{L,a}} \exp \Bigl(-\frac{1}{\varepsilon}\int\limits_{D_L}F(h(z))dz \Bigr)     \nu_1^{\varepsilon}(\mathrm{d}h)\\ 
%&=~\frac{1}{Z^{\varepsilon}_1}\int_{\cH^{\varepsilon}}  \exp \Bigl(-\frac{1}%{\varepsilon}\int_{-\varepsilon^{-\lambda}}^{\varepsilon^{-\lambda}}\int\limits_0^1F(h)dxdy\\
%& \qquad \qquad  -\frac{1}%{\varepsilon}\int_{-\varepsilon^{-\lambda}}^{\varepsilon^{-\lambda}}\int\limits_0^1\frac{1}{2} |\nabla %h|^2dxdy \Bigr)\cL^N(\mathrm{d} h)\ .
&=~\frac{1}{Z^{\varepsilon}_1}\exp\Bigl(-\frac{C_*}{\varepsilon}\Bigr)  \int_{\cH^{L,a}}  \exp \Bigl(-\frac{1}{\varepsilon} \cF(h)\Bigl) \cL^{N}(\mathrm{d}h)\ .
\end{split}
\end{equation}
The main step in the proof will be to give a lower bound on this normalization constant. This is done in Proposition \ref{prop_boundnormalizationconstant} by calculating the integral in (\ref{eqn_zeps}) in a tubular neighborhood of the set of minimizers $M$. To this end recall the following version of the coarea formula:

\begin{lemma}\label{lem_coarea}
Let $f$ be a Lipschitz function $f:A \subseteq E \to I \subseteq \bR$, where $E$ is a $N$-di\-men\-sion\-al Euclidean space and $A$ is an open subset and $I$ some interval. Denote by $\mathcal{L}^N, \mathcal{L}^1$ and $\mathcal{H}^{N-1}$ the Lebesgue measure on $E$, on $\bR$ and the $(N-1)$-dimensional Hausdorff measure on $E$ respectively. Suppose that the gradient (which exists $\mathcal{L}^N$-a.e.) $Df$ does not vanish $\mathcal{L}^N$ a.e. in $A$. Then for every nonnegative measurable test function $\varphi:A \rightarrow \bR$ one has the following formula:
\begin{equation*}\label{eqn_coarea}
\int_A \varphi(x)\mathcal{L}^N(\mathrm{d}x) = \int_I \mathcal{L}^1 (\mathrm{d} \xi) \int_{f^{-1}(\xi)} \mathcal{H}^{N-1}(\mathrm{d}x) \frac{1}{|Df(x)|_E} \varphi(x)\ .
\end{equation*}
\end{lemma}

We want to apply the coarea formula to the Fermi coordinates of $m+L^2(D)$. To this end we need the following

\begin{lemma}\label{lem_gradientfermicoodinate}
 Let $A=\{h\in m+L^2(D)\ \vert\ \distt{h}<\beta\}$ be the set in which Fermi coordinates are defined and let $f:A\to\bR$ be the function given by 
$f(h)=\xi$ when $h$ has coordinates $(\xi,v)$. Then $f$ is Fr\'{e}chet differentiable with
$$Df(m_\xi+v)[w]~=~\frac{\langle \partial_x m_\xi, w\rangle_{L^2}}{\norm{\partial_x m_\xi}_{L^2}^2 - \langle v, \partial^2_x m_\xi \rangle_{L^2} }\ .$$
Furthermore, let $\tilde{f}:\bR^N\to\bR$ be the the composition of $f$ with the embedding $\bR^N\cong\cH^{L,a}\to m+L^2(D)$ obtained by linear interpolation and extension with $\pm1$. Then 
\begin{equation}\label{eqn_gradestimate}
\norm{\nabla \tilde{f}}_{\bR^N}~\leq~2^{d+1}a^{(d+1)/2}\cdot\frac{\norm{\partial_x m_\xi}_{L^2}}{\abs{\norm{\partial_x m_\xi}^2_{L^2} - \langle v, \partial^2_x m_\xi \rangle_{L^2}}}\ .
\end{equation}
\end{lemma}

\begin{proof}
See Lemma 4.2 and Lemma 4.3 of \cite{We10} whose proofs adapt immediately to our higher-dimensional setting.
\end{proof}

We can now give an asymptotic lower bound on the normalization constant $Z^\eps$ defined in (\ref{eqn_zeps}).

\begin{proposition}\label{prop_boundnormalizationconstant}
 Assume $L(\eps)\sim\eps^{-\lambda}$ and $a(\eps)\sim\eps^{\alpha}$ as $\eps\to 0$ where $\alpha,\lambda>0$ and $\lambda+(d+1)\alpha<1$.
 Then $N=\dim\cH^{L,a}\sim\eps^{-\lambda-(d+1)\alpha}$ and we have the following bound~:
\begin{equation*}\label{eqn_boundnormalizationconstant}
 \liminf\limits_{\eps\to 0}\eps\log\left(Z^\eps\right)~\geq~-C_*\ .
\end{equation*}
\end{proposition}

\begin{proof}
Recall that $Z^\eps$ is given by (\ref{eqn_zeps}). 
 %Starting from (\ref{eqn_zeps}) we can write
%\begin{equation*}\label{eqn_zepsbound1}
% Z^\eps~=~\frac{1}{Z^{\varepsilon}_1}\exp\Bigl(-\frac{C_*}{\varepsilon}\Bigr)  \int_{\cH^{\varepsilon}}  \exp %\Bigl(-\frac{1}{\varepsilon} \cF(h)\Bigl) \cL^{N}(\mathrm{d}h)\ .
%\end{equation*}
To find a lower bound on it is sufficient to restrict the integration to a tubular neighborhood of $M$. Let us set $I_\varepsilon:=[-\varepsilon^{-\lambda}+\varepsilon^{-\lambda_1}, \varepsilon^{-\lambda}-\varepsilon^{-\lambda_1}]$ and 
$$ A_\xi~:=~ \left\{h \in \cH^{L,a}\colon h=m_\xi +v \colon   \langle v, \partial_x m_\xi \rangle_{L^2} = 0\ ,\ \| v \|_{H^1}<\delta,\norm{v}_\infty < 1 \right\}\ ,$$
for some $\delta<\delta_3$ is to be determined later. Then we consider a tubular neighborhood of $M$ defined by
$$ A~:=~\bigcup\limits_{\xi\in I_\eps} A_\xi\ .$$
Using the estimate on the energy landscape Proposition \ref{prop_energylandscape} i) we have for $h\in A$ with $h=m_\xi+v$:
$$\exp \Bigl(-\frac{1}{\varepsilon} \cF(h)\Bigl) ~\geq~ \exp \left(-\frac{c_3}{\varepsilon}  \|v\|_{H^1}^2  \right)\ .$$
Note that $v$ is not a piecewise linear function but a general function in $L^2(D)$. We can approximate $v$ by a function $v^\eps\in\cH^{L,a}_0$ by setting $v^\eps(x,y)=h(x,y)-m_\xi^\eps(x,y)$ for $(x,y)\in D_{L,a}$. Using the error bound from Lemma \ref{lem_discbo1} we get
$$\norm{v^\eps-v}_{H^1}~=~\norm{m_\xi^\eps-m_\xi}_{H^1}~\leq~C \varepsilon^{-\lambda_1/2} \varepsilon^{\alpha}\ .$$
Putting this together we get:
\begin{equation}\label{eqn_bound2}
 Z^\eps~\geq~\exp\Bigl(-\frac{C_*}{\varepsilon}\Bigr)\exp\big(-C\eps^{-\lambda_1+2\alpha-1}\big)\frac{1}{Z^{\varepsilon}_1}\int_{A} \exp \left(-\frac{2 c_3}{\varepsilon}  \|v^{\varepsilon} \|_{H^1}^2 \right)   \mathcal{L}^{N}(\mathrm{d}h)\ .
\end{equation}
Using the coarea formula to evaluate the integral over the set $A$ we obtain:
$$\int_{A} \exp \left(-\frac{2 c_3}{\varepsilon}  \|v^{\varepsilon} \|_{H^1}^2   \right) \mathcal{L}^{N}(\mathrm{d}h) = \int_{I_\varepsilon} \mathrm{d}\xi \int_{A_\xi} \frac{1}{|\tilde{\nabla}\tilde{ f}|}\exp \Bigl(-\frac{2c_3}{\varepsilon} \|v^{\varepsilon}\|_{H^1}^2 \Bigl) \mathcal{H}^{N-1}(\mathrm{d}h)\ ,$$
where $\mathcal{H}^{N-1}$ is the codimension $1$ Hausdorff measure on $\cH^{L,a}$. Using equation (\ref{eqn_gradestimate}) and choosing a smaller $\delta$ if necessary we can estimate the gradient uniformly on $A$ by:
$$\frac{1}{|\tilde{\nabla}\tilde{ f}|}~\geq~C\eps^{-\frac{(d+1)\alpha}{2}}\ .$$
Hence we obtain that
\begin{equation}\label{eqn_bound3}
\begin{split}
 \int_{A} &\exp \left(-\frac{2 c_3}{\varepsilon}  \|v^{\varepsilon} \|_{H^1}^2  \right) \mathcal{L}^{N}(\mathrm{d}h)\\
~&\geq~C \eps^{-\lambda-\frac{(d+1)\alpha}{2}} \int_{A_\xi} \exp \left(-\frac{2 c_3}{\varepsilon} \|v^{\varepsilon} \|_{H^1}^2\right)  \mathcal{H}^{N-1}(\mathrm{d}h)\ .
\end{split}
\end{equation}
Let us focus on the last integral over the set $A_\xi$. By a linear change of coordinates one can write 
\begin{equation}\label{eqn_bound4}
 \int_{A_\xi} \exp \left(-\frac{2 c_3}{\varepsilon} \|v^{\varepsilon} \|_{H^1}^2\right)  \mathcal{H}^{N-1}(\mathrm{d}h)~=~\int_{B_\xi} \exp \left(-\frac{2 c_3}{\varepsilon}  \|v \|_{H^1}^2\right)  \mathcal{H}^{N-1}(\mathrm{d}v),
\end{equation}
where $B_\xi= \left\{v \in \cH^{L,a}_0 \colon\ \langle v, \partial_x m_\xi \rangle_{L^2} = \langle m_\xi-m_\xi^{\varepsilon}, \partial_x m_\xi  \rangle_{L^2},\  \| v \|_{H^1}\leq\delta,\norm{v}_\infty \leq 1\right\}$ .
In order to conclude, we need the following lemma:

\begin{lemma} \label{lem_euclidianintegration}
Let $E$ be a finite-dimensional Euclidean space with Lebesgue measure $\mathcal{L}$ and codimension $1$ Hausdorff measure $\mathcal{H}$. Let $a^*= \langle a, \cdot \rangle \in E^*$ be a linear form and  $x \mapsto \langle x, \Sigma x \rangle$ be a symmetric, positive bilinear form. Furthermore, write for $ b \in \bR$ and $\delta,\rho >0$
\begin{equation*}
 \begin{split}
   \tilde{B}^{b,\delta^2,\rho}~&=~ \left\{ x \in E \colon a^* (x)=b \text{ and } \langle x, \Sigma x \rangle \leq \delta^2\ ,\ \norm{x}_\infty<\rho \right\}\ , \\
   \tilde{B}^{\delta^2,\rho}~&=~ \left\{ x \in E \colon \langle x, \Sigma x \rangle \leq \delta^2\ ,\ \norm{x}_\infty<\rho \right\}\ . 
 \end{split}
\end{equation*}

Furthermore, set $l^2= \inf_{x \in \tilde{B}^{b,\infty,\infty}} \langle x, \Sigma x  \rangle$ and let $n$ be a $\Sigma$-unit normal vector on $\tilde{B}^{0,\infty,\infty}$, i.e. $\langle n, \Sigma x \rangle =0 \,$ for all $x \in \tilde{B}^{0,\infty,\infty}$ and $\langle n ,\Sigma n \rangle =1$. We assume that $\norm{n}_\infty\leq 1$. Then one has for every $b$:
\begin{equation*}\label{eqn_linalg}
 \int\limits_{\tilde{B}^{\delta^2-l^2,1-2\delta}}\exp \left( - \langle x,\Sigma x \rangle \right) \mathcal{L}(\mathrm{d}x)~\leq~\dfrac{2 \delta\exp(l^2)}{\langle \Sigma n, \Sigma n \rangle^{\frac{1}{2}}} \int\limits_{\tilde{B}^{b,\delta^2,1}} \exp \left(- \langle x,\Sigma x \rangle \right) \mathcal{H}(\mathrm{d}x)\ .
\end{equation*}
Furthermore, one has the following expressions for $l^2$: 
\begin{equation*}\label{eqn_varpri1}
 l^2=\frac{b^2}{\langle a, \Sigma^{-1} a \rangle }, \qquad  \langle a, \Sigma^{-1} a \rangle= \Big( \sup_{\eta \colon \langle \eta ,\Sigma \eta \rangle =1} a^* (\eta)  \Big)^2\ .
\end{equation*}
The vector $n$ is given as $\pm \frac{\Sigma^{-1}a}{\sqrt{\langle a, \Sigma^{-1} a \rangle } }$ and consequently:
\begin{equation*}\label{eqn_varpri4}
  \langle \Sigma n, \Sigma n \rangle =   \frac{1}{ \langle a, \Sigma^{-1} a \rangle } \left( \sup_{\eta \colon \langle  \eta , \eta \rangle =1} a^* (\eta) \right)^2\ .
\end{equation*}
\end{lemma}
\begin{proof}
Using the coarea formula one can write:
\begin{equation}\label{Linalg3}
 \begin{split}
&\int\limits_{\tilde{B}^{\delta^2-l^2,1-2\delta}} \exp \left( -\langle x,\Sigma x \rangle \right) \mathcal{L}(\mathrm{d}x) \\
&\leq~  \int\limits_{-\delta}^{\delta} \int\limits_{ \tilde{B}^{0,\delta^2-l^2,1-\delta}}  \exp \left(- \langle (y+\lambda n),\Sigma (y+\lambda n) \rangle \right) \dfrac{1}{\langle \Sigma n, \Sigma n \rangle^{\frac{1}{2}}} \mathcal{H}(\mathrm{d}y) \mathrm{d}\lambda\\
&\leq~ \dfrac{2 \delta}{\langle \Sigma n, \Sigma n \rangle^{\frac{1}{2}}}\int\limits_{ \tilde{B}^{0,\delta^2-l^2,1-\delta}}  \exp \left(- \langle y,\Sigma y \rangle \right) \mathcal{H}(\mathrm{d}y)\\
&=~ \dfrac{2 \delta\exp(l^2)}{\langle \Sigma n, \Sigma n \rangle^{\frac{1}{2}}}  \int\limits_{ \tilde{B}^{0,\delta^2-l^2,1-\delta}} \exp \left(- \langle (y +l n),\Sigma (y + l n) \rangle \right) \mathcal{H}(\mathrm{d}y)\\
&\leq~ \dfrac{2 \delta\exp(l^2)}{\langle \Sigma n, \Sigma n \rangle^{\frac{1}{2}}}  \int\limits_{ \tilde{B}^{b,\delta^2,1}} \exp \left(- \langle y ,\Sigma y  \rangle \right) \mathcal{H}(\mathrm{d}y).
\end{split}
\end{equation}
The other assertions are elementary.
\end{proof}

We will apply this lemma in the situation were $E=\cH^{L,a}$, $a^*(v)=\langle v,\partial_x m_\xi\rangle_{L^2}$, $b=\langle m_\xi-m_\xi^\eps,\partial_x m_\xi\rangle_{L^2}$ and $\langle v,\Sigma v\rangle = \frac{2c_3}{\eps}\norm{v}^2_{H^1}$. In this case we must verify that $\norm{n}_\infty\leq1$ and we can estimate the constants appearing in Lemma \ref{lem_euclidianintegration} as follows.
\begin{lemma}\label{lem_ballintegrationconstants}
One has for $\varepsilon$ small enough:
\begin{enumerate}
 \item[(i)] $\norm{n}_\infty \leq 1$ ,
 \item[(ii)] $\langle m_\xi-m_\xi^{\varepsilon}, \partial_x m_\xi  \rangle_{L^2} \leq C \varepsilon^{-\lambda_1/2}\eps^{2\alpha}$ ,
 \item[(iii)] $l^2 \leq C \varepsilon^{4\alpha-\lambda_1-1}$ ,
 \item[(iv)] $\langle \Sigma n, \Sigma n \rangle  \geq C \varepsilon^{-1+(d+1)\alpha}$ .
\end{enumerate}
\end{lemma}

\begin{proof}
 The necessary calculations for (ii)-(iv) can be found in Lemma 4.6 of \cite{We10} and are easily adapted to our setting, hence we only show (i). Note that
\begin{equation*}
 \norm{n}_\infty~\leq~\langle n,n\rangle~=~\frac{\langle\Sigma^{-1}a,\Sigma^{-1}a\rangle}{\langle a,\Sigma^{-1}a\rangle}~\leq~\left(\inf\limits_{u\neq0}\frac{\langle u,\Sigma u\rangle}{\langle u, u\rangle}\right)^{-2}\frac{\langle a, a\rangle}{\langle a,\Sigma^{-1} a\rangle}\ .
\end{equation*}
Using equation (\ref{eqn_concentration2}) we see that the first factor is bounded above by $C\eps^2a(\eps)^{-(d+1)}$, where $a(\eps)$ is the mesh size of the grid. A similar calculation as in the above reference shows that $\langle a,\Sigma^{-1} a\rangle\geq C\eps$. Finally we note that
\begin{equation*}
 \langle a, a\rangle~=~\sup\limits_{\langle u, u\rangle=1}a^*(u)~=~\sup\limits_{\langle u, u\rangle=1}\int_{D_L}\partial_x m_\xi u\,dz~\leq~\int_\bR\partial_x m_\xi dx~=~2~.
\end{equation*}
Hence we obtain $\norm{n}_\infty\leq C \eps a(\eps)^{-(d+1)}=C\eps^{1-(d+1)\alpha}<1$ for $\eps$ small enough, since $\lambda+(d+1)\alpha<1$ by assumption and so $\alpha<\frac{1}{d+1}$.
\end{proof}

Applying the last two lemmata to equation (\ref{eqn_bound4}) we obtain:
\begin{equation}\label{eqn_bound5}
\begin{split}
 &\int\limits_{B_\xi} \exp \left(-\frac{2 c_3}{\varepsilon}  \|v \|_{H^1}^2\right)  \mathcal{H}^{N-1}(\mathrm{d}v)\\~
&\geq~\frac{C}{\delta}\eps^{\frac{-1+(d+1)\alpha}{2}}\exp(-C\eps^{4\alpha-\lambda_1-1})\int\limits_{B} \exp \left(-\frac{2 c_3}{\varepsilon}  \|v \|_{H^1}^2\right)  \mathcal{L}^{N}(\mathrm{d}v)
\end{split}
\end{equation}
where $B= \left\{ v \in \cH_0^{L,a} \colon \frac{2c_3}{\varepsilon} \| v \|_{H^1}^2 \leq \frac{2c_3}{\varepsilon} \delta^2 -l^2, \norm{v}_\infty \leq 1-2\delta \right\}$.
Note that by Lemma \ref{lem_ballintegrationconstants} ii) and (\ref{eqn_gamma1}) we know that $l^2\eps/2c_3\to0$ as $\eps\to0$. Hence the last integral over the set $B$ can be bounded from below by
\begin{equation*}
 %\begin{split}
  %\int_{\left\{\| v \|_{H^1}^2\leq\frac{\delta^2}{2}, \norm{v}_\infty <1-2\delta\right\}} &\exp \left(-\frac{2c_3}{\varepsilon} \|v \|_{H^1}^2 \right) \mathcal{L}^{N}(\mathrm{d}v)\\~&=~ 
 Z^{\varepsilon,\kappa}_3 \rho^{\varepsilon}_\kappa \left( \| v \|_{H^1}^2 \leq \frac{\delta^2}{2}, \norm{v}_\infty<1-2\delta  \right)\ ,
 %\end{split}
\end{equation*}
where $\kappa=4c_3$. The concentration property of the measure $\rho^{\varepsilon}_\kappa$ from Lemma \ref{lem_concentrationMDGFF} shows that this probability can be bounded below by $\frac{1}{2}$ for $\eps$ small enough. Summarizing the estimates (\ref{eqn_bound2}), (\ref{eqn_bound3}), (\ref{eqn_bound4}) and (\ref{eqn_bound5}) we have obtained the following bound:
\begin{equation*}
 Z^\eps~\geq~\exp\Bigl(-\frac{C_*}{\varepsilon}\Bigr)\exp\big(-C\eps^{-\lambda_1+2\alpha-1}\big)\exp(-C\eps^{4\alpha-\lambda_1-1})\eps^{-\lambda-\frac{1}{2}}\frac{Z^{\varepsilon,\kappa}_3}{Z^{\varepsilon}_1}\ .
\end{equation*}
Finally invoking Lemma \ref{lem_normalizationbounds} ii) and assumption (\ref{eqn_gamma1}) finishes the proof.
\end{proof}

We are now in a position to finish the proof of our main theorem. 

\begin{proof}[Proof of Theorem \ref{thm_mainthm}]
We will first show that for $\delta\leq\delta_0$ there holds
\begin{equation}\label{eqn_Zmubound}
 \limsup\limits_{\eps\searrow 0} \eps \log\left(Z^\eps \mu^\eps\big(\distt{h}\geq\delta\big)\right)~\leq~-C_* - c_0\delta^2\ .
\end{equation}
Combining this with the asymptotic of $Z^\eps$ from Proposition \ref{prop_boundnormalizationconstant} then immediately yields the theorem.

So let us denote $A^\delta:=\{h\in\cH^{L,a}\ :\ \distt{h}\geq\delta\}$. From Proposition \ref{prop_energylandscape} ii) we know that 
$$\cF(h)-c_0\delta^2~\geq~0\quad \mbox{on the set } A^\delta\ .$$ 
Hence we can estimate
\begin{align*}\label{eqn_Zmucalculation}
 Z^\eps \mu^\eps(A^\delta)~&=~\exp\left(-\frac{C_*}{\eps}\right)\frac{1}{Z^\eps_1}\int\limits_{A^\delta}\exp\left(-\frac{1}{\eps}\cF(h)\right)\cL^N(dh)\\
			   &\leq~\exp\left(-\frac{1}{\eps}(C_*+c_0\delta^2)\right)\frac{1}{Z^\eps_1}\int\limits_{A^\delta}\exp\left(-\frac{1}{\eps}(\cF(h)-c_0\delta^2)\right)\cL^N(dh)\\
			   &\leq~\exp\left(-\frac{1}{\eps}(C_*+c_0\delta^2)\right)\frac{1}{Z^\eps_1}\int\limits_{A^\delta}\exp\left(-(\cF(h)-c_0\delta^2)\right)\cL^N(dh)\\
			   &\leq~\exp\left(-\frac{1}{\eps}(C_*+c_0\delta^2)\right)e^{-c_0\delta^2}\frac{1}{Z^\eps_1}\int\limits_{A^\delta}
\exp\left(-\frac{1}{2}\norm{\nabla h}^2_{L^2}\right)\cL^N(dh)\\
			   &\leq~\exp\left(-\frac{1}{\eps}(C_*+c_0\delta^2)\right)e^{-c_0\delta^2}\frac{Z^\eps_2}{Z^\eps_1}\ .
\end{align*}
Note that by Lemma \ref{lem_normalizationbounds} $Z_2^\eps/Z^\eps_1\leq\eps^{-N/2}$. Recalling that $N(\eps)\sim \eps^{-\lambda-(d+1)\alpha}$ and the assumption 
$\lambda+(d+1)\alpha<1$ we conclude the estimate (\ref{eqn_Zmubound}) which finishes the proof.
\end{proof}

\bibliographystyle{plain}
\nocite{*}
\bibliography{literatur}

\begin{thebibliography}{10}

\bibitem{AC79}
S.~Allen and J.~Cahn.
\newblock A microscopic theory for antiphase boundary motion and its
  application to antiphase domain coarsening.
\newblock {\em Acta Metallurgica}, 27(6):1085--1095, 1979.

\bibitem{Be08}
L.~Bertini, S.~Brassesco, and P.~Butt\`{a}.
\newblock Dobrushin states in the $\phi^4_1$-model.
\newblock {\em Archive for Rational Mechanics and Analysis}, 190:477--516,
  2008.

\bibitem{Bo87}
C.~Boldrighini, A.~de~Masi, A.~Pelegrinotti, and E.~Presutti.
\newblock Collective phenomena in interacting particle systems.
\newblock {\em Stochastic Processes and their Applications}, 25:137 -- 152,
  1987.

\bibitem{PT07}
G.~da~Prato and L.~Tubaro.
\newblock Wick powers in stochastic {PDE}: an introduction.
\newblock {\em Pubblicazione del Dipartimento di Matematica dell Universit\`{a}
  di Trento}, UTM 771, 2007.

\bibitem{Fu95}
T.~Funaki.
\newblock The scaling limit for a stochastic {PDE} and the separation of
  phases.
\newblock {\em Probab. Theory Related Fields}, 102(2):221--288, 1995.

\bibitem{Il93}
T.~Illmanen.
\newblock Convergence of the {A}llen-{C}ahn equation to {B}rakkes motion by
  mean curvature.
\newblock {\em J. Differential Geom.}, 38(2):417--461, 1993.

\bibitem{KOR07}
R.~Kohn, F.~Otto, M.~G. Reznikoff, and E.~Vanden-Eijnden.
\newblock Action minimization and sharpinterface limits for the stochastic
  {A}llen-{C}ahn equation.
\newblock {\em Comm. Pure Appl. Math.}, 60:393--438, 2007.

\bibitem{OR07}
F.~Otto and M.~G. Reznikoff.
\newblock Slow motion of gradient flows.
\newblock {\em J. Differential Equations}, 237(2):372--420, 2007.

\bibitem{RE05}
M.~Reznikoff and E.~Vanden-Eijnden.
\newblock Invariant measures of stochastic partial differential equations and
  conditioned diffusions.
\newblock {\em C. R. Math. Acad. Sci. Paris}, 340(4):305--308, 2005.

\bibitem{Si74}
B.~Simon.
\newblock {\em The ${P}(\phi)_2$ {E}uclidean (quantum) field theory}.
\newblock Princeton University Press, Princeton, NJ, 1974.

\bibitem{We10}
H.~Weber.
\newblock Sharp interface limit for invariant measures of a stochastic
  {A}llen-{C}ahn equation.
\newblock {\em Comm. Pure Appl. Math.}, 63(8):1071--1109, 2010.

\end{thebibliography}

{\sc Matthias Erbar, Universit\"at Bonn, Institue For Applied Mathematics, Edenicher Allee 60, 53115 Bonn, Germany}\\
{\em E-mail address:} \texttt{erbar@iam.uni-bonn.de}

\end{document}